\documentclass[11pt]{article}
\usepackage{amsmath,amsthm,amssymb}
\usepackage[dvips]{graphicx}

\renewcommand{\P}{\mathbb P}

\newcommand{\dd}{{\operatorname{d}}}
\newcommand{\sn}{{\operatorname{sn}}}
\newcommand{\cn}{{\operatorname{cn}}}
\newcommand{\dn}{{\operatorname{dn}}}
\newcommand{\zn}{{\operatorname{zn}}}
\newcommand{\CAP}{{\operatorname{cap}}}
\newcommand{\C}{C_{\alpha,\beta}}
\newcommand{\laK}{\lambda{K}}
\newcommand{\KK}{\tfrac{K}{2}}

\begin{document}


\title{An Upper Bound for the Logarithmic Capacity of Two Intervals\footnote{published in: Complex Variables and Elliptic Equations {\bf 53} (2008), 65--75.}}
\author{Klaus Schiefermayr\footnote{University of Applied Sciences Upper Austria, School of Engineering and Environmental Sciences, Stelzhamerstrasse\,23, 4600 Wels, Austria, \textsc{klaus.schiefermayr@fh-wels.at}}}
\date{}
\maketitle

\theoremstyle{plain}
\newtheorem{theorem}{Theorem}
\newtheorem{lemma}{Lemma}
\theoremstyle{definition}
\newtheorem*{remark}{Remark}

\begin{abstract}
The logarithmic capacity (also called Chebyshev constant or transfinite diameter) of two real intervals $[-1,\alpha]\cup[\beta,1]$ has been given explicitly with the help of Jacobi's elliptic and theta functions already by Achieser in 1930. By proving several inequalities for these elliptic and theta functions, an upper bound for the logarithmic capacity in terms of elementary functions of $\alpha$ and $\beta$ is derived.
\end{abstract}



\section{Introduction}


Let $C$ be a compact set in the complex plane and let $\P_n$ be the set of all polynomials of degree less or equal $n$. Then $L_n(C)$, defined by
\begin{equation}\label{Ln}
L_n(C):=\inf_{p_{n-1}\in\P_{n-1}}\sup_{z\in{C}}|z^n-p_{n-1}(z)|,
\end{equation}
is usually called the \emph{minimum deviation} of degree $n$ on $C$. The \emph{logarithmic capacity} (or \emph{Chebyshev constant}) of $C$ is then defined by the limit
\begin{equation}
\CAP(C):=\lim_{n\to\infty}\root{n}\of{L_n(C)}.
\end{equation}
The logarithmic capacity of $C$ can be given in a completely different way by
\begin{equation}
\CAP(C)=\lim_{n\to\infty}\sup_{z_i,z_j\in{C}}
\Bigl(\prod_{1\leq{i}<j\leq{n}}|z_i-z_j|\Bigr)^{\frac{2}{n(n-1)}}
\end{equation}
and, in this connection, is also called the \emph{transfinite diameter}. For further reading and weighted analogues concerning this constant, see \cite{MhaskarSaff}.


In this paper, we consider the case of two intervals, i.e.,
\begin{equation} \label{C}
C=\C:=[-1,\alpha]\cup[\beta,1], \quad -1<\alpha<\beta<1.
\end{equation}
The logarithmic capacity of two intervals $\C$ can only be given with the help of Jacobi's elliptic and theta functions (see Theorem\,\ref{Theorem_Achieser}). The purpose of this paper is to give an upper bound for $\CAP(\C)$ in terms of elementary functions of $\alpha$ and $\beta$.


Let $K\equiv{K}(k)$ and $E\equiv{E}(k)$ be the complete elliptic integral of the first and second kind, respectively, where $0<k<1$ is known as the modulus of $K$ and $E$. Let $k':=\sqrt{1-k^2}$, $K'\equiv{K}'(k):=K(k')$, and let $\sn(u)\equiv\sn(u,k)$, $\cn(u)\equiv\cn(u,k)$, $\dn(u)\equiv\dn(u,k)$ be Jacobi's elliptic functions with respect to the modulus $k$. Further, let $\Theta(u)\equiv\Theta(u,k)$, $H(u)\equiv H(u,k)$, $H_1(u)\equiv{H}_1(u,k)$ and $\Theta_1(u)\equiv\Theta_1(u,k)$ be the four theta functions of Jacobi (old notation).


Further, we need Jacobi's zeta function $\zn(u)\equiv\zn(u,k)$ given by
\begin{equation}
\zn(u)=\int_{0}^{u}\dn^2(v)\,\dd{v}-\frac{u\,E}{K}=\frac{\Theta'(u)}{\Theta(u)}.
\end{equation}


For the definitions and many important properties of these functions see, e.g., \cite{ByrdFriedman}, \cite{Lawden}, \cite[chapter\,X]{MagnusOberhettingerSoni} or \cite{WhittakerWatson}.


A formula for the logarithmic capacity of two intervals in terms of Jacobi's elliptic and theta functions has been derived by N.I.\,Achieser\,\cite{Achieser-1930}. For the background of \cite{Achieser-1930} and some related results, see \cite{Zolotarev1}, \cite{Zolotarev2}, \cite{Achieser-1929}, \cite{Achieser-1932}, \cite{Peh1995}, and \cite{PehSch}.


\begin{theorem}[Achieser\,\cite{Achieser-1930}]\label{Theorem_Achieser}
Let $\C:=[-1,\alpha]\cup[\beta,1]$, $-1<\alpha<\beta<1$, let the modulus $k$, $0<k<1$, be given by
\begin{equation} \label{k}
k^2:=\frac{2(\beta-\alpha)}{(1-\alpha)(1+\beta)}
\end{equation}
and let $0<\lambda<1$ be such that
\begin{equation} \label{sn}
\sn^2(\laK)=\frac{1-\alpha}{2}.
\end{equation}
Then
\begin{equation} \label{cap1}
\CAP(\C)=\frac{\Theta^4(0)}{2\,\dn^2(\laK)\Theta^4(\laK)}=\frac{2{k'}^2K^2}{\pi^2\dn^2(\laK)\Theta^4(\laK)}.
\end{equation}
\end{theorem}


The problem in operating with formula\,\eqref{cap1} is the difficult connection of the modulus $k$ and the parameter $\lambda$ with the values $\alpha$ and $\beta$ via formulae \eqref{k} and \eqref{sn}, where the modulus $k$ appear both in the term $K\equiv{K}(k)$ and as the parameter of the elliptic function $\sn(u)\equiv\sn(u,k)$. With the help of some inequalities and monotonicity properties of Jacobi's elliptic and theta functions given and proved in Section\,3, we are able to prove an upper bound for $\CAP(\C)$ in terms of elementary functions of $\alpha$ and $\beta$ in Section\,2. Concluding this section, we collect some useful properties of $\alpha$ and $\beta$ (as functions of $k$ and $\lambda$) and of the capacity $\CAP(\C)$ in the following remark.


\begin{remark}\hfill{}
\begin{enumerate}
\item By \eqref{k} and \eqref{sn},
\begin{equation}\label{cndn}
\cn^2(\laK)=\frac{1+\alpha}{2},\quad\dn^2(\laK)=\frac{1+\alpha}{1+\beta}
\end{equation}
and
\begin{equation}\label{k'}
{k'}^2=\frac{(1+\alpha)(1-\beta)}{(1-\alpha)(1+\beta)}.
\end{equation}
\item By \eqref{k} and \eqref{sn}, there is a one-to-one correspondence between $\{\alpha,\beta\}$, $-1<\alpha<\beta<1$, and $\{k,\lambda\}$, $0<k<1$, $0<\lambda<1$. Moreover, for $\alpha\equiv\alpha(k,\lambda)$ and $\beta\equiv\beta(k,\lambda)$, we get
\begin{equation}\label{AlphaBeta}
\begin{aligned}
\alpha(k,\lambda)&=1-2\,\sn^2(\lambda{K}),\\
\beta(k,\lambda)&=\frac{2\,\cn^2(\lambda{K})}{\dn^2(\lambda{K})}-1=2\,\sn^2((1+\lambda)K)-1.
\end{aligned}
\end{equation}
\item Let $0<k<1$ be fixed. By \eqref{AlphaBeta}, both, $\alpha(k,\lambda)$ and $\beta(k,\lambda)$, are strictly monotone decreasing functions of $\lambda$ with
\[
\lim_{\lambda\to0}\alpha(k,\lambda)=\lim_{\lambda\to0}\beta(k,\lambda)=1,\quad
\lim_{\lambda\to1}\alpha(k,\lambda)=\lim_{\lambda\to1}\beta(k,\lambda)=-1.
\]
\item Let $0<\lambda<1$ be fixed. By \eqref{AlphaBeta} and \cite[Sect.\,1 \& 4]{CarlsonTodd}, $\alpha(k,\lambda)$ and $\beta(k,\lambda)$ is a strictly monotone decreasing and increasing function of $k$, respectively. Moreover,
\[
\lim_{k\to0}\alpha(k,\lambda)=\lim_{k\to0}\beta(k,\lambda)=\cos(\lambda\pi),\quad
-\lim_{k\to1}\alpha(k,\lambda)=\lim_{k\to1}\beta(k,\lambda)=1,
\]
i.e., for the limiting case $k\to0$, we get $\C\to[-1,1]$, and for $k\to1$, $\C$ vanishes.
\item By \eqref{AlphaBeta}, for $\lambda=\frac{1}{2}$, $0<\lambda<\frac{1}{2}$, $\frac{1}{2}<\lambda<1$, there is $\alpha+\beta=0$, $\alpha+\beta>0$, $\alpha+\beta<0$, respectively.
\item If $\lambda$ changes to $1{-}\lambda$, then $\{\alpha,\beta\}$ changes to $\{-\beta,-\alpha\}$ and the capacity remains the same, i.e.,
\begin{equation}\label{cap2}
\CAP(\C)=\frac{\Theta^4(0)}{2\,\dn^2((1-\lambda)K)\Theta^4((1-\lambda)K)}\
=\frac{2{k'}^2K^2}{\pi^2\dn^2((1-\lambda)K)\Theta^4((1-\lambda)K)}.
\end{equation}
\item The capacity is monotone, i.e., for $-1<\alpha_1\leq\alpha_2<\beta_2\leq\beta_1<1$,
\[
\CAP(C_{\alpha_1,\beta_1})\leq\CAP(C_{\alpha_2,\beta_2}).
\]
\item Let $C$ be any compact set in the complex plane, then for $a>0$
\begin{equation}
\CAP(a\,C)=a\,\CAP(C).
\end{equation}
\item By (vii), for $\alpha+\beta\geq0$, there is a trivial lower bound for $\CAP(\C)$ given by
\begin{equation}\label{LB1}
\CAP(\C)\geq\CAP([-1,-\beta]\cup[\beta,1]=\tfrac{1}{2}\sqrt{1-\beta^2}.
\end{equation}
\item By (vii), for $\alpha+\beta\geq0$ and $\alpha<0$, there is a trivial upper bound for $\CAP(\C)$ given by
\begin{equation}
\CAP(\C)\leq\CAP([-1,\alpha]\cup[-\alpha,1])=\tfrac{1}{2}\sqrt{1-\alpha^2}.
\end{equation}
Again by (vii), another trivial upper bound for all $-1<\alpha<\beta<1$ is
\begin{equation}
\CAP(\C)\leq\CAP([-1,1])=\frac{1}{2}.
\end{equation}
\item Another lower bound can be extracted from \cite{Pommerenke}
\begin{equation}\label{LB2}
\CAP(\C)\ge\CAP([-1,\alpha])+\CAP([\beta,1])=\frac{1+\alpha}{4}+\frac{1-\beta}{4}.
\end{equation}
\item Finally, another upper bound can be derived from an inequality of Gillis\,\cite{Gillis}, which, in our special case of two intervals, reads as follows:
\begin{equation}\label{UB1}
\CAP(\C)\leq2\exp\Bigl(\frac{\log\frac{1+\alpha}{8}\,\log\frac{1-\beta}{8}}{\log\frac{(1+\alpha)(1-\beta)}{64}}\Bigr)
\end{equation}
\item Let $\alpha\in(-1,1)$ be fixed. Then, by \eqref{LB2} and \eqref{UB1},
\begin{equation}
\lim_{\beta\to1}\CAP(\C)=\frac{1+\alpha}{4}=\CAP([-1,\alpha]).
\end{equation}
Note that this is not true for the minimum deviation $L_n(\C)$.
\item Let $\varphi,\psi\in(0,\pi)$ such that $\cos\varphi=\alpha$, $\cos\psi=\beta$, and let
\begin{equation}
A:=\bigl\{z\in\mathbb{C}:|z|=1,\arg{z}\in[-\psi,\psi]\cup[\varphi,2\pi-\varphi]\bigr\},
\end{equation}
i.e., $\C$ is the projection of $A$ onto the real axis. Then, by a result of Robinson\,\cite{Robinson},
\begin{equation}
\CAP(A)=\sqrt{2\,\CAP(\C)}.
\end{equation}
\end{enumerate}
\end{remark}

\section{Main Results}


Let us first give a lower bound for the logarithmic capacity of two intervals $\C$, proved by Alexander Yu.\ Solynin\,\cite[Sect.\,2.2]{Solynin} (in fact he derived a lower bound for the logarithmic capacity of several intervals). In our notation, his result reads as follows:


\begin{theorem}[Solynin\,\cite{Solynin}]
Let $-1<\alpha<\beta<1$ and define
\[
\varphi:=\arccos(\alpha), \qquad \psi:=\arccos(\beta),
\]
then
\begin{equation}\label{LB-Solynin}
\CAP(\C)\geq\frac{1}{2}\max_{\delta\in[\psi,\varphi]}\Bigl[
\Bigl(\sin\bigl(\frac{\psi\pi}{2\delta}\bigr)\Bigr)^{2\delta^{2}/\pi^{2}}
\Bigl(\sin\bigl(\frac{(\pi-\varphi)\pi}{2(\pi-\delta)}\bigr)\Bigr)^{2(\pi-\delta)^{2}/\pi^{2}}\Bigl].
\end{equation}
\end{theorem}


\begin{remark}
Numerical calculations suggest that the optimal $\delta$ in \eqref{LB-Solynin}, which cannot be calculated analytically, may be roughly approximated by $\delta=(\psi+\varphi)/2$. With the optimal $\delta$, the lower bound \eqref{LB-Solynin} is excellent, but also for $\delta=(\psi+\varphi)/2$, the lower bound \eqref{LB-Solynin} is very good.
\end{remark}


In addition to \eqref{LB-Solynin}, we give another lower bound for the capacity of $\C$ in very simple functions of $\alpha$ and $\beta$.


\begin{theorem}
Let $-1<\alpha<\beta<1$, then
\begin{equation}\label{LB3}
\CAP(\C)\geq\frac{\sqrt{k'}}{1+k'}=\frac{\sqrt[4]{(1-\alpha^2)(1-\beta^2)}}
{\sqrt{(1-\alpha)(1+\beta)}+\sqrt{(1+\alpha)(1-\beta)}},
\end{equation}
where equality is attained for $\alpha+\beta=0$.
\end{theorem}
\begin{proof}
Let $0<k<1$ be fixed, then, by \eqref{cap1} and Lemma\,\ref{Lemma_DnTheta}, $\CAP(\C)$ is strictly monotone decreasing in $\lambda$, $0\leq\lambda\leq\frac{1}{2}$. Thus, by Lemma\,\ref{Lemma_ThetaK2} and $\dn(\KK)=\sqrt{k'}$,
\[
\CAP(\C)\geq\frac{2{k'}^2K^2}{\pi^2\dn^2(\KK)\,\Theta^4(\KK)}=\frac{\sqrt{k'}}{1+k'}.
\]
For $\frac{1}{2}<\lambda<1$, the inequality follows by \eqref{cap2}.
\end{proof}


\begin{remark}
For any $\alpha,\beta$, $-1<\alpha<\beta<1$, $\alpha+\beta\geq0$, the lower bound in \eqref{LB3} is greater (i.e.\ better) than the lower bound in \eqref{LB1}.
\end{remark}


The next theorem contains the main result of the paper, an upper bound for the logarithmic capacity of two intervals $\C$.


\begin{theorem}\label{Thm_main}
Let $-1<\alpha<\beta<1$ and let $k\in(0,1)$ and $\lambda\in(0,1)$ be given by \eqref{k} and \eqref{sn}, respectively, then
\begin{equation}\label{UB2}
\CAP(\C)\leq\frac{1}{2}\,\dn^2(\laK)\,\exp\Bigl(2(E/K-{k'}^2)\,\log^2\frac{1+\sn(\laK)}{\cn(\laK)}\Bigr),
\end{equation}
where equality is attained for the limiting case $\lambda\to0$.
\end{theorem}
\begin{proof}
\begin{align*}
\CAP(\C)&=\frac{1}{2\,\dn^2((1-\lambda)K)}\Bigl(\frac{\Theta(0)}{\Theta((1-\lambda)K)}\Bigr)^4\quad\text{by \eqref{cap2}}\\
&\leq\frac{{k'}^2}{2\,\dn^2((1-\lambda)K)}\,\exp\bigl(2(E/K-{k'}^2)\lambda^2K^2\bigr)\quad\text{by Lemma \ref{Lemma_BoundTheta}}\\
&\leq\frac{1}{2}\,\dn^2(\laK)\,\exp\bigl(2(E/K-{k'}^2)\log^2\frac{1+\sn(\laK)}{\cn(\laK)}\bigr)\quad\text{by Lemma \ref{Lemma_BoundCnDn}}
\end{align*}
\end{proof}


\begin{remark}\hfill{}
\begin{enumerate}
\item Upper bound \eqref{UB2} is only good in case of $\alpha+\beta\geq0$, i.e., $0<\lambda\leq\frac{1}{2}$. For $\alpha+\beta<0$, one has to replace $\{\alpha,\beta\}$ by $\{-\beta,-\alpha\}$.
\item Analogously to \eqref{UB2}, we also derived a lower bound for $\CAP(\C)$, which turned out to be less than the lower bound \eqref{LB-Solynin}, thus we dismissed it.
\item In Figure\,1, we compare the capacity $\CAP(\C)$ with the upper bounds \eqref{UB2} and \eqref{UB1}. The graphs of these three functions are plotted for $\alpha\in\{\pm0.7,\pm0.4,\pm0.1\}$ and $|\alpha|\leq\beta\leq{1}$.
\item By \eqref{k}, \eqref{sn}, \eqref{cndn}, and \eqref{k'}, inequality\eqref{UB2} of Theorem\,\ref{Thm_main} may be written in the form
\begin{equation}\label{UB3}
\CAP(\C)\leq\frac{1+\alpha}{2(1+\beta)}\exp\Bigl[2\Bigl(\frac{E}{K}-\frac{(1+\alpha)(1-\beta)}{(1-\alpha)(1+\beta)}\Bigr)
\log^2\frac{\sqrt{2}+\sqrt{1-\alpha}}{\sqrt{1+\alpha}}\Bigr]
\end{equation}
In this bound, there appear $K\equiv{K}(k)$ and $E\equiv{E}(k)$, which cannot be expressed by elementary functions of $k$. Nevertheless, there exists a couple of inequalities in the literature, see \cite{Alzer}, \cite{AlzerQiu}, \cite{AVV1990}, \cite{AVV1992}, \cite{AVV}, \cite{CG}, and \cite{QV}, involving elementary functions of $k$, from which we collect the best ones in the following lemma. Hence, inequality \eqref{UB3}, together with \eqref{ineq_K} and \eqref{ineq_E}, using \eqref{k} and \eqref{k'}, gives an upper bound for $\CAP(\C)$ \emph{in terms of elementary functions} of $\alpha$ and $\beta$.
\end{enumerate}
\end{remark}


\begin{figure}[ht]
\begin{center}
\includegraphics[scale=0.8]{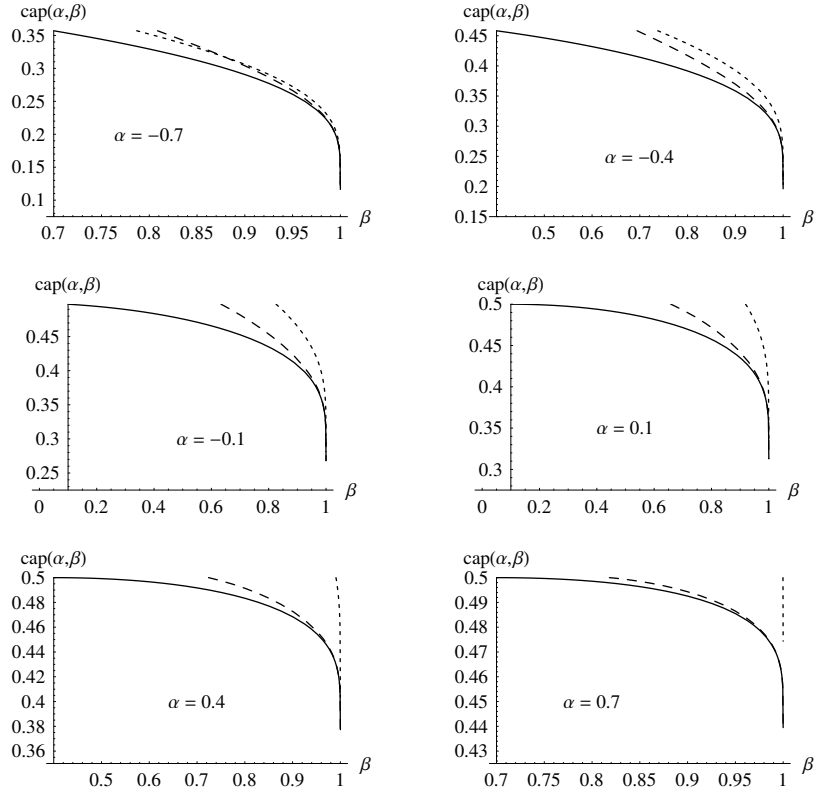}
\caption{\label{Fig} The capacity $\CAP(\C)$ (solid line) with the upper bounds \eqref{UB2} (dashed line) and \eqref{UB1} (dotted line) for $\alpha\in\{-0.7,-0.4,-0.1,0.1,0.4,0.7\}$ and $|\alpha|\leq\beta\leq{1}$}
\end{center}
\end{figure}


\begin{lemma} \label{Lemma_KE}
Let $0<k<1$.
\begin{enumerate}
\item For the complete elliptic integral of the first kind, we have the inequalities
\begin{equation}\label{ineq_K}
\max\bigl\{K_1,K_2\bigr\}<K(k)<\min\bigl\{K_3,K_4,K_5\bigr\},
\end{equation}
where
\begin{equation}
\begin{aligned}
K_1&:=\frac{\pi}{2}\Bigl(\frac{\tanh^{-1}(k)}{k}\Bigr)^{3/4},\\
K_2&:=\Bigl(1+\frac{{k'}^2}{4}\Bigr)\log\Bigl(\frac{4}{k'}\Bigr)-\frac{{k'}^2}{4},
\end{aligned}
\end{equation}
and
\begin{equation}
\begin{aligned}
K_3&:=\frac{\pi}{2}\Bigl(\frac{\frac{3}{4}\log(k')}{k'-1}+\frac{1}{2(1+k')}\Bigr),\\
K_4&:=\log\Bigl(\frac{4}{k'}+\bigl(e^{\pi/2}-4\bigr)(k')^{\gamma}\Bigr),\quad \gamma:=\frac{4-\frac{\pi}{4}\,e^{\pi/2}}{e^{\pi/2}-4},\\
K_5&:=\Bigl(1+\frac{{k'}^2}{4}\Bigr)\log\Bigl(\frac{4}{k'}\Bigr).
\end{aligned}
\end{equation}
\item For the complete elliptic integral of the second kind, we have the inequalities
\begin{equation}\label{ineq_E}
E_1<E(k)<E_2,
\end{equation}
where
\begin{equation}
E_1:=\frac{\pi}{2}\Bigl(\frac{1+{k'}^{3/2}}{2}\Bigr)^{2/3},\quad
E_2:=\frac{\pi}{2}\Bigl(\frac{1+{k'}^{\delta}}{2}\Bigr)^{1/\delta},\quad \delta:=\frac{\log(2)}{\log(\pi/2)}.
\end{equation}
\end{enumerate}
\end{lemma}
\begin{proof}\hfill{}
\begin{enumerate}
\item $K_1<K(k)$ is proved in \cite[Theorem\,18]{AlzerQiu}, $K_2<K(k)$ can be found in \cite[equation\,(31)]{CG}. For $K(k)<K_3$ and $K(k)<K_4$ see Theorems\,19 and 20 of \cite{AlzerQiu}, $K(k)<K_5$ is proved in \cite{Alzer}.
\item Both inequalities are proved by Alzer and Qiu in \cite[Theorem\,22]{AlzerQiu}.
\end{enumerate}
\end{proof}


\begin{remark}
Numerical computations show that
\begin{equation}
\max\bigl\{K_1,K_2\bigr\}=
\begin{cases}
K_1\quad&\text{for }0<k\leq0.888\ldots\\
K_2&\text{for }0.888\ldots\leq{k}<1
\end{cases}
\end{equation}
and
\begin{equation}
\min\bigl\{K_3,K_4,K_5\bigr\}=
\begin{cases}
K_3\quad&\text{for }0<k\leq0.971\ldots\\
K_4\quad&\text{for }0.971\ldots\leq{k}\leq0.990\ldots\\
K_5&\text{for }0.990\ldots\leq{k}<1
\end{cases}
\end{equation}
\end{remark}

\section{Auxiliary Results for Jacobi's Elliptic and Theta Functions}


\begin{lemma}\label{Lemma_ThetaK2}
Let $0<k<1$. Then
\begin{equation}
\begin{aligned}
\Theta(\KK)&=\Theta_1(\KK)=\sqrt[4]{\tfrac{2}{\pi^2}(1+k')K^2\sqrt{k'}},\\
H(\KK)&=H_1(\KK)=\sqrt[4]{\tfrac{2}{\pi^2}(1-k')K^2\sqrt{k'}}.\\
\end{aligned}
\end{equation}
\end{lemma}
\begin{proof}
For the four theta-functions, the following duplication formulas hold (see \cite[p.\,21]{Lawden} or formula (1051.25) of \cite{ByrdFriedman})
\begin{equation}\label{Theta2u}
\begin{aligned}
H(2u)\,\Theta(0)\,H_1(0)\,\Theta_1(0)&=2H(u)\,\Theta(u)\,H_1(u)\,\Theta_1(u)\\
H_1(2u)\,H_1(0)\,\Theta^2(0)&=\Theta^2(u)\,H_1^2(u)-H^2(u)\,\Theta_1^2(u)\\
\Theta_1(2u)\,\Theta_1(0)\,\Theta^2(0)&=\Theta^2(u)\,\Theta_1^2(u)-H^2(u)\,H_1^2(u)\\
\Theta(2u)\,\Theta^3(0)&=\Theta_1^4(u)-H_1^4(u)=\Theta^4(u)-H^4(u)
\end{aligned}
\end{equation}
Further, by $\Theta(-u)=\Theta(u)$ and $\Theta_1(u+K)=\Theta(u)$,
\[
\Theta(\KK)=\Theta(-\KK)=\Theta_1(-\KK+K)=\Theta_1(\KK),
\]
and, by $H(u)=-H(-u)$ and $H_1(u+K)=-H(u)$,
\[
H(\KK)=-H(-\KK)=H_1(-\KK+K)=H_1(\KK).
\]
Hence, by \eqref{Theta2u}, setting $u=\KK$,
\begin{gather*}
\Theta^4(\KK)-H^4(\KK)=\Theta(K)\,\Theta^3(0)=\frac{4{k'}^{3/2}K^2}{\pi^2},\\
2\,\Theta^2(\KK)\,H^2(\KK)=H(K)\,\Theta(0)\,H_1(0)\,\Theta_1(0)=\frac{4kK^2\sqrt{k'}}{\pi^2}.
\end{gather*}
From these two equalities, the asserted formulas follow after some computation.
\end{proof}


\begin{lemma} \label{Lemma_DnTheta}
Let $0<k<1$. Then the function $\dn^2(u)\,\Theta^4(u)$ is strictly monotone increasing in $u$, $0\leq{u}\leq\KK$.
\end{lemma}
\begin{proof}
Differentiating the above function gives
\[
\frac{\partial}{\partial{u}}\bigl\{\dn^2(u)\,\Theta^4(u)\bigr\}
=2\,\dn^2(u)\,\Theta^4(u)\Bigl(\underbrace{2\,\zn(u)-\frac{k^2\sn(u)\,\cn(u)}{\dn(u)}}_{=:g(u)}\Bigr).
\]
Note that $g(0)=g(\KK)=0$. We prove that $g(u)>0$ for $u\in(0,\KK)$ by showing that $g''(u)<0$ for $u\in(0,\KK)$. Indeed,
\[
g''(u)=-\frac{2k^4\sn(u)\,\cn(u)}{\dn^3(u)}\bigl(\cn(u)\,\dn(u)+k'\sn(u)\bigr)
\bigl(\underbrace{\cn(u)\,\dn(u)-k'\sn(u)}_{=:h(u)}\bigr)
\]
and $h(u)$ is strictly monotone decreasing for $u\in(0,\KK)$ and $h(\KK)=0$, thus $h(u)>0$, hence $g''(u)<0$ for $u\in(0,\KK)$. This completes the proof.
\end{proof}


\begin{lemma}\label{Lemma_MonCnDn}
Let $0<k<1$, then, for $0<u<K$, $\frac{\partial}{\partial{k}}\,\{\cn(u,k)\}>0$ and $\frac{\partial}{\partial{k}}\,\{\dn(u,k)\}<0$.
\end{lemma}
\begin{proof}
By (710.52) and (710.53) of \cite{ByrdFriedman},
\begin{align*}
\frac{\partial}{\partial{k}}\,\bigl\{\cn(u)\bigr\}&=\frac{\sn(u)\,\dn(u)}{k{k'}^2}\Bigl(
\underbrace{E(u)-{k'}^2u-\frac{k^2\sn(u)\,\cn(u)}{\dn(u)}}_{=:f(u)}\Bigr)\\
\frac{\partial}{\partial{k}}\,\bigl\{\dn(u)\bigr\}&=\frac{k\,\sn(u)\,\cn(u)}{{k'}^2}\Bigl(
\underbrace{E(u)-{k'}^2u-\frac{\sn(u)\,\dn(u)}{\cn(u)}}_{=:g(u)}\Bigr)
\end{align*}
Hence, it suffices to show that $f(u)>0$ and $g(u)<0$ for $0<u<K$. Obviously, $f(0)=g(0)=0$, and, using the formulas (121.00) of \cite{ByrdFriedman},
\[
f'(u)=\frac{k^2{k'}^2\sn^2(u)}{\dn^2(u)},\qquad
g'(u)=-\frac{{k'}^2}{\cn^2(u)},
\]
thus $f'(u)>0$ and $g'(u)<0$ for $0<u<K$, which gives $f(u)>0$ and $g(u)<0$, $0<u<K$.
\end{proof}


\begin{lemma}\label{Lemma_BoundCnDn}
Let $0<k<1$, then, for $0\leq{u}<K$,
\begin{equation}
\frac{1}{\dn(u,k)}\leq\cosh(u)\leq\frac{1}{\cn(u,k)},
\end{equation}
and
\begin{equation}
\log\frac{1+k\,\sn(u)}{\dn(u)}\leq{u}\leq\log\frac{1+\sn(u)}{\cn(u)}.
\end{equation}
\end{lemma}
\begin{proof}
By Lemma\,\ref{Lemma_MonCnDn} and the fact that
\[
\lim_{k\to1}\cn(u,k)=\lim_{k\to1}\dn(u,k)=\frac{1}{\cosh(u)}
\]
we get $\cn(u,k)\leq1/\cosh(u)$ and $\dn(u,k)\geq1/\cosh(u)$ for every $u\in(0,K)$. The second inequality follows immediately by the well known formula $\cosh^{-1}(u)=\log(u+\sqrt{u^2-1})$.
\end{proof}


\begin{lemma}\label{Lemma_BoundZn}
Let $0<k<1$, then, for $0\leq{u}\leq{K}$,
\begin{equation}
\zn(u)\leq(E-{k'}^2K)(1-u/K).
\end{equation}
\end{lemma}
\begin{proof}
Define
\[
f(u):=(E-{k'}^2K)(1-u/K)-\zn(u),
\]
then $f(K)=0$ and $f'(u)={k'}^2-\dn^2(u)<0$ for $0\leq{u}\leq{K}$, thus $f(u)\geq0$ for $0\leq{u}\leq{K}$.
\end{proof}


\begin{lemma}\label{Lemma_BoundTheta}
Let $0<k<1$, then, for $0\leq\lambda\leq1$,
\begin{equation} \label{ThetaUpperBound}
\frac{\Theta(0)}{\Theta(\laK)}\leq\sqrt{k'}\,\exp\bigl(\tfrac{1}{2}\bigl(E/K-{k'}^2\bigr)(1-\lambda)^2K^2\bigr).
\end{equation}
\end{lemma}
\begin{proof}
For $0\leq{u}\leq{K}$, we have
\[
\log\frac{\Theta(K)}{\Theta(u)}=-\log\frac{\Theta(u)}{\Theta(K)}=-\int_{K}^{u}\zn(v)\,\dd{v}=\int_{u}^{K}\zn(v)\,\dd{v}.
\]
Thus, by Lemma\,\ref{Lemma_BoundZn},
\[
\log\frac{\Theta(K)}{\Theta(u)}\leq\int_{u}^{K}(E-{k'}^2K)(1-v/K)\,\dd{v}
=(E-{k'}^2K)(K/2-u+u^2/(2K))
\]
Taking $\exp$ on both sides and using the relation $\Theta(0)=\sqrt{k'}\,\Theta(K)$ gives \eqref{ThetaUpperBound}.
\end{proof}


\bibliographystyle{amsplain}

\bibliography{UpperBoundCapacity}

\end{document}